\documentclass[11pt,reqno]{amsart}
\usepackage[utf8]{inputenc}
\usepackage{a4wide}
\usepackage{appendix}

\usepackage{graphicx}
\usepackage{enumerate}
\usepackage{hyperref}
\usepackage{cleveref}
\usepackage{amsmath,amsopn,amssymb,amsfonts,stmaryrd}
\usepackage{verbatim}
\usepackage{amsthm}
\usepackage{mathtools}
\usepackage{color}
\usepackage{enumitem}
\usepackage[framemethod=TikZ]{mdframed}
\usepackage{bbm}
\usepackage{mathrsfs}
\usepackage{booktabs}
\usepackage{caption}
\usepackage{bm}
\usepackage{tensor}

\makeatletter
\@namedef{subjclassname@2020}{%
	\textup{2020} Mathematics Subject Classification}
\makeatother

\newcommand{\IR}{{\mathbb R}}

\newcommand\restr[2]{\ensuremath{\left.#1\right|_{#2}}}

\theoremstyle{plain}
\newtheorem{thm}{Theorem}[section]
\newtheorem{cor}[thm]{Corollary}
\newtheorem{lem}[thm]{Lemma}
\newtheorem{prop}[thm]{Proposition}
\newtheorem{fact}[thm]{Fact}
\newtheorem{facts}[thm]{Facts}
\newtheorem{rem}[thm]{Remark}

\newtheorem{MainThm}{Theorem}

\newtheorem{Mthm}[MainThm]{Theorem}
\newtheorem{Mcor}[MainThm]{Corollary}
\newtheorem{Mprop}[MainThm]{Proposition}
\newtheorem{Mlem}[MainThm]{Lemma}

\theoremstyle{definition}
\newtheorem{example}[thm]{Example}
\newtheorem{defn}[thm]{Definition}

\newtheorem{ques}{Question}

\usepackage{aliascnt}
\usepackage{cleveref}

\theoremstyle{plain} 
\newaliascnt{cor}{thm}
\newtheorem{cor}[cor]{Corollary}
\aliascntresetthe{cor}
\crefname{cor}{corollary}{corollaries}
\Crefname{cor}{Corollary}{Corollaries}

\newaliascnt{lem}{thm}
\newtheorem{lem}[lem]{Lemma}
\aliascntresetthe{lem}
\crefname{lem}{lemma}{lemmas}
\Crefname{lem}{Lemma}{Lemmas}

\newaliascnt{prop}{thm}
\newtheorem{prop}[prop]{Proposition}
\aliascntresetthe{prop}
\crefname{prop}{proposition}{propositions}
\Crefname{prop}{Proposition}{Propositions}

\newaliascnt{conj}{thm}
\newtheorem{conj}[conj]{Conjecture}
\aliascntresetthe{conj}
\crefname{conj}{conjecture}{conjectures}
\Crefname{conj}{Conjecture}{Conjectures}

\newtheorem{ques}[ques]{Question}

\crefname{ques}{question}{questions}
\Crefname{ques}{Question}{Questions}

\newaliascnt{fact}{thm}
\newtheorem{fact}[fact]{Fact}
\aliascntresetthe{fact}
\crefname{fact}{fact}{facts}
\Crefname{fact}{Fact}{Facts}

\newaliascnt{facts}{thm}
\newtheorem{facts}[facts]{Facts}
\aliascntresetthe{facts}
\crefname{facts}{fact}{facts}
\Crefname{facts}{Fact}{Facts}

\newaliascnt{rem}{thm}
\newtheorem{rem}[rem]{Remark}
\aliascntresetthe{rem}
\crefname{rem}{remark}{remarks}
\Crefname{rem}{Remark}{Remarks}

\theoremstyle{definition} 
\newaliascnt{example}{thm}
\newtheorem{example}[example]{Example}
\aliascntresetthe{example}
\crefname{example}{example}{examples}
\Crefname{example}{Example}{Examples}

\newaliascnt{defn}{thm}
\newtheorem{defn}[defn]{Definition}
\aliascntresetthe{defn}
\crefname{defn}{definition}{definitions}
\Crefname{defn}{Definition}{Definitions}

\crefname{clai}{claim}{claims}
\Crefname{clai}{Claim}{Claims}

\theoremstyle{plain} 

\newaliascnt{Mthm}{MainThm}
\newtheorem{Mthm}[Mthm]{Theorem}
\aliascntresetthe{Mthm}
\crefname{Mthm}{theorem}{theorems}
\Crefname{Mthm}{Theorem}{Theorems}

\newaliascnt{Mcor}{MainThm}
\newtheorem{Mcor}[Mcor]{Corollary}
\aliascntresetthe{Mcor}
\crefname{Mcor}{corollary}{corollaries}
\Crefname{Mcor}{Corollary}{Corollaries}

\newaliascnt{Mlem}{MainThm}
\newtheorem{Mlem}[Mlem]{Lemma}
\aliascntresetthe{Mlem}
\crefname{Mlem}{lemma}{lemmas}
\Crefname{Mlem}{Lemma}{Lemmas}

\newaliascnt{Mprop}{MainThm}
\newtheorem{Mprop}[Mprop]{Proposition}
\aliascntresetthe{Mprop}
\crefname{Mprop}{proposition}{propositions}
\Crefname{Mprop}{Proposition}{Propositions}

\theoremstyle{plain} 

\crefname{Ithm}{theorem}{theorems}
\Crefname{Ithm}{Theorem}{Theorems}

\crefname{Icor}{corollary}{corollaries}
\Crefname{Icor}{Corollary}{Corollaries}

\crefname{Ilem}{lemma}{lemmas}
\Crefname{Ilem}{Lemma}{Lemmas}

\crefname{Iprop}{proposition}{propositions}
\Crefname{Iprop}{Proposition}{Propositions}

\theoremstyle{definition} 

\crefname{Idef}{definition}{definitions}
\Crefname{Idef}{Definition}{Definitions}

\numberwithin{equation}{section}

\def\a{\alpha}

\def\d{\delta}

\def\k{\kappa}

\def\s{\sigma}

\def\a{\alpha}

\def\d{\delta}

\def\k{\kappa}

\def\s{\sigma}

\def\RR{{\mathbb R}}

\def\d{{\mathrm{d}}}

\def\R{{R_{\alpha}}}

\def\SS{{\mathbb{S}}}

\newcommand{\gc}{\gamma}
\newcommand{\dgc}{\dot{\gc}}

\makeatletter
\newcommand{\vast}{\bBigg@{2}}
\newcommand{\Vast}{\bBigg@{5}}
\makeatother

\newcommand{\RNum}[1]{\uppercase\expandafter{\romannumeral #1\relax}}
\allowdisplaybreaks

\title[Growth rate of prime periodic magnetic geodesics]{The growth rate of closed prime magnetic geodesics\\ on closed contact manifolds} 
\author{Lina Deschamps}
\address{Faculty of Mathematics and Computer Science,
	University of Heidelberg,
	Im Neuenheimer Field 205,
	69120 Heidelberg, Germany}
\email{ldeschamps@mathi.uni-heidelberg.de}

\author{Levin Maier}
\address{Faculty of Mathematics and Computer Science,
	University of Heidelberg,
	Im Neuenheimer Field 205,
	69120 Heidelberg, Germany}
\email{lmaier@mathi.uni-heidelberg.de}

\author{Tom Stalljohann}
\address{Faculty of Mathematics and Computer Science,
	University of Heidelberg,
	Im Neuenheimer Field 205,
	69120 Heidelberg, Germany}
\email{tstalljohann@mathi.uni-heidelberg.de}

\keywords{}

\subjclass[2020]{37J45, 37J55, 53D25}

\begin{document}
	\maketitle
\renewcommand{\abstractname}{Abstract}
	\begin{abstract}
In this paper, we prove that for any given closed contact manifold, there exists an infinite-dimensional space of Riemannian metrics which can be identified with the space of bundle metrics on the induced contact distribution. 
For each such metric, and for all energy levels, the number of periodic orbits of the corresponding magnetic geodesic flow grows at least as fast as the number of geometrically distinct periodic Reeb orbits of period less than~$t$.

As a corollary, we deduce that for every closed $3$–manifold which is not a \emph{graph manifold}, there exists an open $C^1$–neighborhood of the set of nondegenerate contact forms such that for each contact form in this neighborhood, there exists an infinite-dimensional space of Riemannian metrics as above. 
For the corresponding magnetic systems, the number of prime closed magnetic geodesics grows at least exponentially on all energy levels. 
Consequently, the restriction of the magnetic geodesic flow to any energy surface has positive topological entropy.
	\end{abstract}	
\section{Introduction}
 
The problem of counting closed geodesics has historically 
been a major driving force in geometry and the calculus of variations, 
particularly in establishing results on their growth rate with respect to length. This classical question has also sparked growing interest in the study of the growth rates of Reeb flows and magnetic geodesic flows, 
both of which can be regarded as natural generalizations of the geodesic flow.

\subsection*{Reeb flows}
Let $M$ be a closed manifold of dimension $\dim M = 2n + 1$, equipped with a one-form $\alpha \in \Omega^1(M)$ such that $\alpha \wedge (\mathrm{d}\alpha)^n$ is a volume form.  
Then $(M,\alpha)$ is called a closed contact manifold.  
There exists a unique vector field $\R \in \Gamma(TM)$ satisfying
\begin{equation}\label{eq: def Reeb vector field}
    \alpha(\R) = 1, 
\qquad 
\mathrm{d}\alpha(\R, \cdot) = 0,
\end{equation}
called the \emph{Reeb vector field} of $(M,\alpha)$.  
Its flow, denoted by $\varPhi_{\R}^t$, is called the \emph{Reeb flow} of $(M,\alpha)$.  
We denote by 
$\mathcal{R}_t(M,\alpha)$ the number of geometrically distinct Reeb orbits of period less than $t$.\\ Before moving on, note that a natural subclass of Reeb flows consists of the geodesic flows of Finsler metrics $F$ on $M$, since, after a possible reparametrization, the geodesic flow on the unit sphere bundle $S_FM$ is precisely the Reeb flow of the Hilbert contact form, where the geodesic spray is the Reeb vector field; see, for example,~\cite{Gg08} for further details.

\subsection*{Magnetic geodesic flows}
We present the mathematical framework used to study the dynamics of a charged particle in the presence of a magnetic field, following V. Arnold's pioneering approach~\cite{ar61}.
    
Let $(M,g)$ be a closed, connected Riemannian manifold and $\sigma\in\Omega^2(M)$ be a closed two-form. The form $\sigma$ is called \emph{magnetic field} and $(M,g,\sigma)$ a \emph{magnetic system}. This triple determines the skew-symmetric bundle endomorphism $Y\colon TM\to TM$, the \emph{Lorentz force}, by
\begin{equation}\label{e:Lorentz}
    g_q\left(Y_qu,v\right)=\sigma_q(u,v),\qquad \forall\, q\in M,\ \forall\,u,v\in T_qM.
\end{equation}
We call a smooth curve $\gc\colon \RR\to M$ a \emph{magnetic geodesic} of $(M,g,\sigma)$ if it satisfies \begin{equation}\label{e:mg}
		\nabla_{\dgc}\dgc= Y_{\gc}\dgc
\end{equation}
where $\nabla$ denotes the Levi-Civita connection of the metric $g$. The equation~\eqref{e:mg} reduces to the geodesic equation \(\nabla_{\dot{\gamma}} \dot{\gamma} = 0\) when \(\sigma = 0\), that is, when the magnetic form vanishes.

Like standard geodesics, magnetic geodesics have constant kinetic energy $E(\gamma,\dot\gamma):=\tfrac12g_\gamma(\dot\gamma,\dot\gamma)$, and hence travel at constant speed $|\dot\gamma|:=\sqrt{g_\gamma(\dot\gamma,\dot\gamma)}$;  since the Lorentz force $Y$ is skew-symmetric.

This conservation of energy reflects the Hamiltonian nature of the system. Indeed, the \emph{magnetic geodesic flow} is defined on the tangent bundle by
\[
\varPhi_{g,\sigma}^t\colon TM\to TM,\quad (q,v)\mapsto \left( \gc_{q,v}(t),\dgc_{q,v}(t)\right),\quad \forall t\in\RR,
\]
where $\gc_{q,v}$ is the unique magnetic geodesic with initial condition $(q,v)\in TM$. For more details on this and magnetic system in general we refer to \Cref{ss: Intermezzo magnetic systems}.

However, a key difference from standard geodesics is that magnetic geodesics with different speeds are not mere reparametrizations of unit-speed magnetic geodesics. This can be seen, for instance, from the fact that the left-hand side of \eqref{e:mg} scales quadratically with the speed, while the right-hand side of \eqref{e:mg} scales only linearly. Therefore, one of the main points of interest is to work out the similarities and differences between standard and magnetic geodesics by varying the kinetic energy.\\

Since magnetic geodesics cannot, in general, be reparametrized to unit speed, 
one must keep track of the energy and consider the growth rate of prime closed magnetic geodesics at a fixed energy level $\k\in(0,\infty)$. 
We call two magnetic geodesics $\gamma_1$ and $\gamma_2$ of $(M,g,\sigma)$ \textit{geometrically equivalent} if $\gamma_2(t) = \gamma_1 (\, t \, + \tau)$ for some $\tau \in \IR \,$.
For fixed energy $\kappa \in (0,\infty)$ we denote by
\[
    \mathfrak{P}^{\kappa}(M,g,\sigma)
    :=
    \bigl\{
       [ \gamma ] \;\big|\;
        \gamma \text{ is a closed prime magnetic geodesic of } (M,g,\sigma)
        \text{ with energy } \kappa
    \bigr\}
\]
the set of geometrical equivalence classes of closed prime magnetic geodesics of $(M,g,\sigma)$ with energy $\kappa \,$.
We then define the \emph{growth function} by
\[
    \mathcal{P}_t^{\kappa}(M,g,\sigma)
    :=
    \#\Bigl\{
        [\gamma] \in \mathfrak{P}^{\kappa}(M,g,\sigma)
        \;\Big|\;
        \ell(\gamma) < t
    \Bigr\} \, ,
\]
where $\ell(\gamma)$ denotes the $g$-length of the closed curve $\gamma \,$.

\subsection{Main result}

With this notation in place, we are now ready to state the main theoretical advance of this work.

\begin{thm}\label{Thm: growth rate of magnetic geodesics on contact mnfds}
Let \((M, \alpha)\) be a closed contact manifold with Reeb vector field \(\R\).  
We define an infinite-dimensional space \(\mathcal{G}\) of Riemannian metrics on $M$, namely the space of all Riemannian metrics \(g\) such that
\[
    g(\R, \R) = 1 
    \quad \text{and} \quad 
    \R \perp_g \ker(\alpha) \, .
\]
Then, for each \(g \in \mathcal{G}\), the exact magnetic system \((M, g, \mathrm{d}\alpha)\) satisfies the following properties:
\begin{enumerate}[label=(\alph*)]
    \item\label{it: 1 growth magnetic geodesics on contact mnfds}
    Every constant reparametrization of a Reeb orbit of $(M, \alpha)$ is both a geodesic of $(M, g)$ and a magnetic geodesic of $(M, g, \mathrm{d}\alpha)$.

    \item\label{it: 2 growth magnetic geodesics on contact mnfds}
    In particular, for every energy level $\kappa \in (0, \infty)$, 
    the magnetic geodesic flow of $(M, g, \mathrm{d}\alpha)$ admits at least as many embedded periodic orbits 
    as the Reeb flow of $(M, \alpha)$ admits Reeb orbits; that is,
    \[
        \mathcal{P}_t^{\kappa}(M, g, \mathrm{d}\alpha) 
        \;\geq\;
        \mathcal{R}_t(M, \alpha)
        \quad 
        \forall\, \kappa, t \in (0, \infty) \, .
    \]

    \item\label{it: 3 growth magnetic geodesics on contact mnfds}
    If, in addition, the Reeb flow admits at least one closed Reeb orbit, then the Mañé critical value $c(M,g,\d\a)$ defined as in \eqref{eq: strict mane value} of the magnetic system $(M, g, \mathrm{d}\alpha)$ is given by
    \[
        c(M, g, \mathrm{d}\alpha) 
        = \tfrac{1}{2} \, \Vert \alpha \Vert_{\infty}^2 
        := \tfrac{1}{2} \max_{x \in M} |\alpha_x|_g^2 
        = \tfrac{1}{2} \, .
    \]
\end{enumerate}
\end{thm}
\begin{rem}\label{rem: bundle metric equal G}
   The space of Riemannian metrics $\mathcal{G}$ in \Cref{Thm: growth rate of magnetic geodesics on contact mnfds} can be identified with the space of bundle metrics $\mathrm{Met}(\xi)$ on the contact distribution $\xi := \ker \alpha$, where we refer to \Cref{s: proof of main thm} for details.
\end{rem}

\begin{rem}
   Moreover, if the Reeb orbit considered in \Cref{Thm: growth rate of magnetic geodesics on contact mnfds}\ref{it: 3 growth magnetic geodesics on contact mnfds} is null-homologous, then the strict Mañé critical value $c_0$---for which we refer, for example, to \cite[§2]{DMS25Contact} for a definition---coincides with the Mañé critical value $c$. 
    Thus, we partially recover \cite[Cor.~H]{DMS25Contact}.
\end{rem}

Before proving this statement, we first illustrate it through growth results 
for magnetic geodesics in this setting.  
Afterwards, we briefly comment on how this result relates to~\cite{DMS25Contact}.  \medskip

\subsection{Illustrations of \Cref{Thm: growth rate of magnetic geodesics on contact mnfds}} 

As already mentioned, the problem of counting closed prime geodesics has historically 
been a driving force in geometry and in the calculus of variations, 
particularly in establishing results concerning their growth rate with respect to length. 
It is therefore natural to investigate this question in the context of magnetic systems as well. 
This will be the main application of \Cref{Thm: growth rate of magnetic geodesics on contact mnfds}; 
see \Cref{cor: growth rate magnetic systems on Sd} and \Cref{cor: exponetial growth on 3d contact manifolds}. 
We begin our discussion with the case of the three-sphere.

\begin{cor}\label{cor: growth rate magnetic systems on Sd}
    Let $M \in \{T^1S^2, S^3\}$, where in the case of $T^1S^2$, 
    the contact form $\alpha$ is the Hilbert contact form 
    of a reversible Finsler metric on $\SS^2$,
    and in the case of $S^3$, it satisfies the assumptions of 
    \cite[Thm.~1.5]{albach2025quadraticgrowthgeodesicstwosphere}. \\[0.5em]
    Then, for each Riemannian metric $g \in \mathcal{G}$, with $\mathcal{G}$ as in 
    \Cref{Thm: growth rate of magnetic geodesics on contact mnfds}, 
    the growth function $\mathcal{P}_t^{\kappa}(M, g, \d\alpha)$ grows at least quadratically in $t$ 
    for every energy level $\kappa \in (0,\infty)$; that is,
    \[
        \liminf_{t \to \infty} 
        \frac{\log \bigl(\mathcal{P}_t^{\kappa}(M, g, \d\alpha)\bigr)}{\log(t)} 
        \;\geq\; 2 \qquad \forall\, \kappa \in (0,\infty)\, .
    \]
\end{cor}

\begin{proof}
   This follows immediately from \Cref{Thm: growth rate of magnetic geodesics on contact mnfds} \ref{it: 2 growth magnetic geodesics on contact mnfds}, 
   combined with the recent result of Albach~\cite[Thm.~1.1, Thm.~1.5]{albach2025quadraticgrowthgeodesicstwosphere}. 
\end{proof}
\begin{rem}
    From \Cref{cor: growth rate magnetic systems on Sd} it follows immediately that, 
    for each of the contact forms $\alpha$ considered therein and for every 
    $g$ in the infinite-dimensional space of Riemannian metrics $\mathcal{G}$, 
    the number of prime geodesics of $(\SS^3,g)$ of length less than $t$ 
    grows quadratically in $t$.
\end{rem}
Lastly, we derive another consequence of \Cref{Thm: growth rate of magnetic geodesics on contact mnfds} 
in light of the so-called \emph{2–or–infinity conjecture} for closed contact $3$–manifolds $(M^3,\alpha)$. 
For a precise overview we refer to \cite[§1.4.]{DMS25Contact} 
and the references therein.  

If $(M^3,\alpha)$ belongs to one of the broad classes described in 
\cite[Cor.~1.7]{cristofarogardiner2024proofhoferwysockizehndersinfinityconjecture} or in \cite[Thm.~1.2]{ColinBrokenbook}, 
then the Reeb flow of $(M^3,\alpha)$ admits infinitely many periodic orbits.  
Combining this with \Cref{Thm: growth rate of magnetic geodesics on contact mnfds} yields:

\begin{cor}\label{cor: growth and 2 infty conj}
    Let $(M^3,\alpha)$ be a closed $3$–manifold belonging to the class described in 
    \cite[Cor.~1.7]{cristofarogardiner2024proofhoferwysockizehndersinfinityconjecture} or in \cite[Thm.~1.2]{ColinBrokenbook}.  
    Let $\mathcal{G}$ denote the infinite-dimensional space of Riemannian metrics as in 
    \Cref{Thm: growth rate of magnetic geodesics on contact mnfds}. \\
    Then, for each $g \in \mathcal{G}$, the magnetic geodesic flow of $(M,g,\d\alpha)$ 
    admits infinitely many embedded periodic orbits on every energy level.
\end{cor}
In fact, this result can be strengthened for generic contact forms on $3$–manifolds $M$ that are not \emph{graph manifolds}. 
In this case, we see that the magnetic geodesic flow exhibits exponential growth of prime geodesics on all energy levels, as well as positive topological entropy.
\begin{cor}\label{cor: exponetial growth on 3d contact manifolds}
    Let $M^3$ be a closed $3$–manifold that is not a graph manifold.  
    
   Then there exists an open $C^1$-neighborhood $\mathcal{U}$ of the set of nondegenerate contact forms on $M$ such that, for each $\alpha \in \mathcal{U}$, there exists an infinite-dimensional space of Riemannian metrics $\mathcal{G}$, as in \Cref{Thm: growth rate of magnetic geodesics on contact mnfds}, with the property that for each $g \in \mathcal{G}$ and for all $\kappa \in (0,\infty)$, the restriction of the magnetic geodesic flow of $(M^3,g,\d\alpha)$ to the energy surfaces $\Sigma_{\kappa}$, defined as in \eqref{eq: energy surface}, exhibits exponential growth. \\ More precisely, there exists $c > 0$ such that
    \[
        \liminf_{t \to \infty} 
        \frac{1}{t} \log \Bigl( \mathcal{P}_t^{\kappa}(M^3,g,\d\alpha) \Bigr) 
        \;\geq\; c 
        \qquad \forall\, \kappa \in (0,\infty)\, .
    \]
    In particular the toplogical entropy $h_{\mathrm{top}}(\varPhi^t_{g,\d\a})$ of the  restriction of the magnetic geodesic flow $\varPhi^t_{g,\d\a}$  onto the energy surface $\Sigma_\k$ is positive for all levels of  the energy $\k\in(0,\infty)$. 
\end{cor}
\begin{rem}
    From \Cref{cor: exponetial growth on 3d contact manifolds} it follows immediately that, 
    for each of the contact forms $\alpha$ considered therein and for every 
    $g$ in the infinite-dimensional space of Riemannian metrics $\mathcal{G}$, 
    the number of prime geodesics of $(M^3,g)$ of length less than $t$ 
    grows exponentially in $t$.
\end{rem}
\begin{rem}
    Note also that neither \Cref{cor: growth rate magnetic systems on Sd} nor \Cref{cor: growth and 2 infty conj} automatically follows from \Cref{cor: exponetial growth on 3d contact manifolds}, since in both cases the manifolds considered are not necessarily graph manifolds, and the contact forms involved are not necessarily nondegenerate. 
\end{rem}
\begin{proof}
By~\cite[Thm.~1.3]{ColinBrokenbook}, there exists an open $C^1$–neighborhood of the set of nondegenerate Reeb
vector fields on $(M, \alpha)$ such that every Reeb vector field in this neighborhood has positive topological entropy; see~\cite{KatokHasselblatt1995} for a definition and background on topological entropy.\\
For flows in dimension~3, positive topological entropy implies that the number of periodic orbits with period less than a
positive number $t$ grows exponentially with $t$. That is, there exists a constant $c > 0$ such that
\begin{equation}\label{eq: exponential growth of periodic reeb orbits}
     \liminf_{t \to \infty} 
        \frac{1}{t} \log \bigl(\mathcal{R}_t(M, \alpha)\bigr)
        \;\geq\; c \, .
\end{equation}
Hence, \eqref{eq: exponential growth of periodic reeb orbits} in combination with \Cref{Thm: growth rate of magnetic geodesics on contact mnfds}\ref{it: 2 growth magnetic geodesics on contact mnfds} yields 
\begin{equation}\label{eq: exponential growth of prime magnetic geodesics}
      \liminf_{t \to \infty} 
        \frac{1}{t} \log \bigl(\mathcal{P}_t^{\kappa}(M^3, g, \mathrm{d}\alpha)\bigr) 
        \;\geq\; c 
        \qquad \forall\, \kappa \in (0,\infty) \, . 
\end{equation}
In particular, by classical results of Katok, Newhouse and Yomdin (see~\cite[Part.1§3]{KatokHasselblatt1995}), we have
\[
h_{\mathrm{top}}\left(\restr{\varPhi^t_{g,\sigma}}{\Sigma_\kappa}\right)
    \;\geq\;
    \liminf_{t \to \infty} 
        \frac{1}{t} \log \bigl(\mathcal{P}_t^{\kappa}(M^3, g, \mathrm{d}\alpha)\bigr)
    \qquad \forall\, \kappa \in (0,\infty) \, .
\]
This, in combination with~\eqref{eq: exponential growth of prime magnetic geodesics}, completes the proof.
\end{proof}
\vspace{0pt}
\subsection{Related results.}
Although much progress has been made over the past thirty years, the dynamics of magnetic flows $(M,g,\s)$ below the Mañé critical value $c(g,\sigma)$ are still not well understood. For example, the existence of a single periodic orbit for $\kappa < c(g,\sigma)$ is known only for almost every energy level (see \cite{AssBenLust16}). Recently, the authors obtained various results on multiplicity in \cite{DMS25Contact}, which were previously available only when $M$ is two-dimensional (see \cite{AbbMacMazzPat17, AsselleMazzucchelli2019, AsselleBenedetti2022, GinzburgGurel2009, GinzburgGurelMacarini2015}) or in a few special cases (e.g.\ \cite{ABM, AsselleMazzucchelli2019Waist, Kerman1999}). The aforementioned multiplicity results in~\cite{DMS25Contact} concern only the existence of finitely many closed prime magnetic geodesics for all energy levels. 

In contrast, this paper contributes to this question by showing, in \Cref{cor: growth and 2 infty conj}, \Cref{cor: growth rate magnetic systems on Sd}, and \Cref{cor: exponetial growth on 3d contact manifolds}, the existence of infinitely many prime periodic magnetic geodesics for all energy levels.
In the case of \Cref{cor: exponetial growth on 3d contact manifolds}, this result can be sharpened to establish exponential growth of closed prime magnetic geodesics on all energy levels. To the best of the author’s knowledge, this represents the first result of its kind in this direction.\newline

\textbf{Relation to \cite{DMS25Contact}}
We begin by noting that the space of Riemannian metrics considered in 
\Cref{Thm: growth rate of magnetic geodesics on contact mnfds} 
is more rigid and smaller than the space of metrics used in the analogous statements of 
\cite{DMS25Contact}.  
Indeed, by the construction in \cite[§ 4.3-§ 4.4]{DMS25Contact}, 
those metrics are only locally prescribed in a small neighborhood of the Reeb orbits, 
while away from the Reeb orbits they can be chosen completely arbitrarily; 
see \cite[Rem. 1.6]{DMS25Contact}
for a precise discussion.  
Consequently, the metrics constructed there do not necessarily satisfy 
\eqref{eq: metric so that each reeb orbit is geodesic.}.  
It remains unclear whether the construction of \cite{DMS25Contact} 
allows one, in the case where the Reeb flow of a closed contact manifold $(M,\alpha)$ 
admits infinitely many periodic orbits, to find a Riemannian metric $g$ such that 
\emph{all} Reeb orbits are magnetic geodesics of $(M,g,\d\alpha)$.  
This issue is closely related to the next point.  

On the other hand, by restricting our attention to the smaller class of Riemannian metrics 
in \Cref{Thm: growth rate of magnetic geodesics on contact mnfds}, 
we obtain stronger dynamical statements for the resulting magnetic systems.  
For instance, compare \Cref{cor: growth rate magnetic systems on Sd} and 
\Cref{cor: growth and 2 infty conj} with the analogous statement in 
\cite{DMS25Contact}, see in particular \cite[Cor. J.2.]{DMS25Contact}.\\

\noindent
\textbf{Acknowledgments:}
The authors thank all participants of the symplectic geometry seminar in Heidelberg, especially P.~Albers, for their valuable comments on related work. We are also grateful to F.~Schlenk for helpful discussions on dynamical richness, to R.~Siefring for many stimulating conversations on Reeb dynamics and pseudoholomorphic curves, and to B.~Albach for sharing unpublished work and for many helpful discussions on the quadratic growth of closed geodesics.
\\
The authors acknowledge funding from the Deutsche Forschungsgemeinschaft (DFG, German Research Foundation) – 281869850 (RTG 2229), 390900948 (EXC-2181/1), and 281071066 (TRR 191). L.D. and L.M. gratefully acknowledge the excellent working conditions and the stimulating interactions at the Erwin Schrödinger International Institute for Mathematics and Physics in Vienna during the thematic programme “Infinite-dimensional Geometry: Theory and Applications”, where part of this work was carried out. L.M. also acknowledges the University of Neuchâtel for their warm hospitality, where part of this work was completed.
\section{Preliminaries: magnetic systems and Ma\~n\'e's critical value}
\subsection{Intermezzo: magnetic systems}\label{ss: Intermezzo magnetic systems} 
For a given magnetic system $(M,g,\sigma)$, we already mentioned that the conservation of energy for magnetic geodesics of $(M,g,\sigma)$ reflects the Hamiltonian nature of the system. Indeed, recall that the \emph{magnetic geodesic flow} is defined on the tangent bundle by
\[
\varPhi_{g,\sigma}^t \colon TM \to TM, \quad (q,v) \mapsto \left( \gc_{q,v}(t), \dgc_{q,v}(t) \right), \quad \forall t \in \RR,
\]
where $\gc_{q,v}$ is the unique magnetic geodesic with initial condition $(q,v) \in TM$.
As shown in~\cite{Gin}, this flow is Hamiltonian with respect to the kinetic energy $E \colon TM \to \RR$ and the twisted symplectic form
\[
\omega_\sigma = \d\lambda - \pi_{TM}^*\sigma,
\]
where $\lambda$ is the metric pullback of the canonical Liouville $1$-form from $T^*M$ to $TM$ via the metric $g$, and $\pi_{TM} \colon TM \to M$ is the basepoint projection. For $\sigma = 0$, we recover in this picture the Hamiltonian formulation of the geodesic flow of $(M,g)$.

Furthermore, by its Hamiltonian nature, the level sets of the kinetic Hamiltonian are invariant under the magnetic geodesic flow \(\varPhi^t_{g,\sigma}\). More precisely:

For a general magnetic system, the zero section corresponds to the set of rest points of the flow. For all energy levels \( \kappa \in (0, \infty) \), all of which are regular values of the kinetic Hamiltonian \( E \), the corresponding hypersurfaces
\begin{equation}\label{eq: energy surface}
    \Sigma_{\kappa} := E^{-1}(\kappa),
\end{equation}
called the \emph{energy surfaces}, are invariant under the magnetic geodesic flow.
Thus, the dynamics of \(\varPhi^t_{g,\sigma}\) restrict to each energy surface \(\Sigma_\kappa\).

In the case of an exact magnetic system $(M,g,\d\alpha)$, that is, when the magnetic field is an exact two-form, the magnetic geodesic flow admits a Lagrangian formulation, and thus also a variational formulation: The corresponding magnetic Lagrangian is
\begin{align*}
    L \colon TM \to \mathbb{R}, \quad L(q, v) := \tfrac{1}{2} |v|^2 - \alpha_q(v) .
\end{align*}
The magnetic geodesic flow \( \Phi^t_{g,\sigma} \) coincides with the Euler--Lagrange flow \( \Phi_L \) associated with the magnetic Lagrangian $L$, see~\cite{Gin}. That is, a curve \( \gamma \colon [0, T] \to M \) is a magnetic geodesic if and only if it is a critical point of the action functional \( S_L \)
\[
S_L(\gamma) := \int_0^T L(\gamma(t), \dot\gamma(t))\,\mathrm{d}t
\]
among all curves \( \delta \colon [0, T] \to M \) with \( \delta(0) = \gamma(0) \) and \( \delta(T) = \gamma(T) \).

This variational principle prescribes the length \( T \) of the time interval but leaves the energy of \( \gamma \) free. On the other hand, for any given energy level \( \kappa \in \mathbb{R} \), the curve \( \gamma \) is a magnetic geodesic with energy \( \kappa \) if and only if it is a critical point of the action functional \( S_{L+\kappa} \) among all curves \( \delta \colon [0, T'] \to M \) such that \( \delta(0) = \gamma(0) \) and \( \delta(T') = \gamma(T) \), for some arbitrary \( T' > 0 \).
\subsection{Mañé's critical value.} This variational formulation underlies the definition of the \emph{Mañé's critical value}, introduced in the seminal works~\cite{CIPP,Man}. This quantity can be interpreted as energy level marking significant dynamical and geometric transitions in the Euler--Lagrange flow induced by the magnetic Lagrangian \( L \). \\

The \emph{Mañé critical value} is 
\begin{equation}\label{eq: strict mane value}
    c(L) := 
    \inf\left\{ \kappa \in \mathbb{R} \,\middle|\, S_{L+\kappa}(\gamma) \geq 0 \,\,\, \forall\, T > 0,\ \forall\, \gamma \in C^\infty(\mathbb{R}/T\mathbb{Z}, M) \right\}
\end{equation}

It can also be defined geometrically, see~\cite{CIPP}, as the smallest energy value containing the graph of an exact one-form on $M$:
\begin{equation}\label{d:mane1}
	c(L) = \inf_{f} \sup_{q\in M} H(q,\d f_q),
\end{equation}
where the infimum is taken over all exact one-forms $\d f$ on $M$ and $H$ is the magnetic Hamiltonian given by the Legendre dual of $L$, that is
\[
H \colon T^*M \to \mathbb{R}, \quad H(q,p) := \tfrac{1}{2}\lvert p + \alpha_q \rvert^2_q \, .\]
For $\k > c(L)$, the level set $H^{-1}(\kappa)$ encloses the Lagrangian graph $\mathrm{gr}(\d f)$ for some exact one-form $\d f$ and is therefore non-displaceable by Gromov’s theorem~\cite{Gr85}. Hence also the energy hypersurface $\Sigma_\kappa = E^{-1}(\kappa)$ is non-displaceable.\\

Finally, we note that the Mañé critical value can also be defined for non-exact magnetic fields, following the works~\cite{CFP10,Me09}, though this generalization lies beyond the scope of this paper. 
\section{The Proof of \Cref{Thm: growth rate of magnetic geodesics on contact mnfds}.}\label{s: proof of main thm}

Given the contact manifold $(M,\a) \,$, with contact structure $\xi = \ker \alpha $ and Reeb vector field $\R$.
For the reader’s convenience, we recall that $\mathcal{G}$ is the space of Riemannian metrics on $M$ satisfying
\begin{equation}\label{eq: metric so that each reeb orbit is geodesic.}
    g(\R,\R) = 1 
    \quad \text{and} \quad 
    \R \perp_g \ker(\alpha) \,\, .
\end{equation}
First we notice that there is a bijection between $\mathcal{G}$ and the space $\mathrm{Met}(\xi)$ of smooth bundle metrics on $\xi \,$. Indeed, any smooth bundle metric $\widetilde{g} \in \mathrm{Met}(\xi)$ extends uniquely to a smooth bundle metric $g$ on $TM$ satisfying \eqref{eq: metric so that each reeb orbit is geodesic.}. This yields a map $\mathrm{Met}(\xi) \rightarrow \mathcal{G}$ whose inverse is given by restricting a metric in $\mathcal{G}$ to $\xi$. \\
Since $\mathrm{Met}(\xi)$ is infinite dimensional (and actually a convex cone), we see that $\mathcal{G}$ is infinite dimensional. (Which in fact proves \Cref{rem: bundle metric equal G}.)\\

Next a result of Sullivan~\cite{Sullivan1978} states that the Reeb flow of $R_\a$ is \emph{geodesible}, i.e.\ there exists a Riemannian metric $g$ on $M$ with respect to which all Reeb orbits are unit-speed geodesics of $(M,g)$.  
In fact, any Riemannian metric $g \in \mathcal{G}$ has this property. For a simplified proof of this statement, we refer the reader to~\cite[Prop.~3.3]{Geiges2022}.  

For each $g \in \mathcal{G}$, by it's definition (see~\eqref{e:Lorentz}), the Lorentz force $Y$ of the magnetic system $(M,g,\d\alpha)$ vanishes on $\ker \d\alpha$, which is, by~\eqref{eq: def Reeb vector field}, precisely the line bundle $\langle R_\alpha \rangle$ spanned by the Reeb vector field $\R$. Thus, the Lorentz force $Y$ of $(M,g,\d\alpha)$ vanishes on $\langle R_\alpha \rangle$. \\
Combining this with the fact that any Reeb orbit $\gamma$ of $\R$ is also a geodesic of $(M,g)$, we conclude that
\[
    \nabla_{\dot{\gamma}} \dot{\gamma} = 0 = Y_{\gamma}\dot{\gamma} \,\, .
\]
Hence every Reeb orbit $\gamma$ is simultaneously a geodesic of $(M,g)$ and a magnetic geodesic of $(M,g,\d\alpha)$. Additionally, as integral curve of the nowhere vanishing vector field $R_\alpha \,$, the Reeb orbit $\gamma$ is automatically embedded. 
Thus $\gamma$ is a magnetic geodesic of geodesic type of $(M,g,\d\alpha)$ in the sense of 
\cite[Def. 1.1.]{DMS25Contact}.\\

Statements~\ref{it: 1 growth magnetic geodesics on contact mnfds} and 
\ref{it: 2 growth magnetic geodesics on contact mnfds} then follow directly from 
\cite[Lemma. 3.1.]{DMS25Contact}, using also that constant reparametrizations of magnetic geodesics of geodesic type are again magnetic geodesics of geodesic type.

Statement \ref{it: 3 growth magnetic geodesics on contact mnfds} 
follows from a computation 
along the lines of the proof of \cite[Thm. B]{DMS25Contact}, 
together with the fact that, by~\eqref{eq: metric so that each reeb orbit is geodesic.}, $\alpha$ is the $g$-metric dual of $\R$ and thus 
the $g$-norm of $\alpha$ attains its maximum along Reeb orbits of $\R$, where it is equal to~1. Which finishes the proof.

\bibliographystyle{abbrv}
	\bibliography{ref}
\end{document}